\documentclass[12pt]{amsart}
\usepackage{amssymb}
\usepackage{txfonts}
\usepackage{amsmath,amssymb,amsbsy,amsfonts,amsthm,latexsym,
                        amsopn,amstext,amsxtra,euscript,amscd,mathrsfs,color,bm,cite}
\usepackage{float}
\usepackage[english]{babel}
\usepackage{mathtools}
\usepackage{todonotes}
\usepackage{url}
\usepackage[colorlinks,linkcolor=blue,anchorcolor=blue,citecolor=blue,backref=page]{hyperref}
\usepackage{cases}

\newcommand{\half}{\frac{1}{2}}
\newcommand{\thalf}{\tfrac{1}{2}}

\newcommand{\sump}{\mathop{{\sum}^{+}}}

\newcommand{\sums}{\mathop{{\sum}^{*}}}

\def\({\left(}
\def\){\right)}
\let\vp=\varphi
\let\ve=\varepsilon
\def\ppmod{\!\!\!\!\!\pmod}

\let\ve=\varepsilon

\let\ol=\overline

\let\vp=\varphi

\theoremstyle{definition}

\theoremstyle{plain}
\newtheorem{theorem}{Theorem}

\newtheorem{lemma}[theorem]{Lemma}

\numberwithin{equation}{section}
\numberwithin{theorem}{section}
\newtheorem{proposition}[theorem]{Proposition}
\def\qed{\ifhmode\textqed\fi
   \ifmmode\ifinner\quad\qedsymbol\else\dispqed\fi\fi}
\def\textqed{\unskip\nobreak\penalty50
    \hskip2em\hbox{}\nobreak\hfil\qedsymbol
    \parfillskip=0pt \finalhyphendemerits=0}
\def\dispqed{\rlap{\qquad\qedsymbol}}

\makeatletter
\@namedef{subjclassname@2010}{%
	\textup{2010} Mathematics Subject Classification}
\makeatother

\usepackage[defaultcolor=red]{changes}
\begin{document}
\title[TBA]{N\MakeLowercase{on-vanishing of} D\MakeLowercase{irichlet} $L$\MakeLowercase{-functions at the central point}}

\author[X. Qin] {Xinhua Qin}

\address{X. Qin: School of Mathematics, Hefei University of Technology, Hefei 230009, P.R. China}
\email{qinxh@mail.hfut.edu.cn}

\author[X. Wu]{Xiaosheng Wu}
\address{X. Wu: School of Mathematics, Hefei University of Technology, Hefei 230009, P.R. China}
\email{xswu@amss.ac.cn}

\begin{abstract}
We prove that for at least $\frac{7}{19}$ of the primitive Dirichlet characters $\chi$ with large general modulus, the central value $L(\frac12,\chi)$ is non-vanishing.
\end{abstract}


\maketitle

\section{Introduction}\label{sec 1}
\noindent
The non-vanishing of various $L$-functions at the central point on the critical line (i.e., the point with real part $\frac12$) has attracted significant interest and generated extensive literature. It is widely conjectured that an $L$-function does not vanish at the central point unless there is some arithmetic significance or the functional equation forces the central value to be zero. The mollifier method has proven remarkably successful in establishing the non-vanishing of $L$-functions. Notable examples include the works of Kowalski et al. \cite{KMVa,KMVb}, Iwaniec and Sarnak \cite{IS2000}, Soundararajan \cite{Sou2000}, and the works discussed below.

In this paper, we investigate the classical non-vanishing problem of primitive Dirichlet $L$-functions. It is conjectured that every primitive Dirichlet $L$-function does not vanish at the central point, a problem that remains unresolved. Consider the family of $L$-functions
\[
\{L(s,\chi): \chi\ \text{is primitive modulo}\ q\}.
\]
A straightforward application of the Cauchy-Schwarz inequality shows
\[
\sums_{\substack{\chi\ppmod{q}\\ L(\frac12,\chi)\neq 0}}1\ge \frac{\left|\sums_{\chi\ppmod{q}}L(\frac12,\chi)\right|^2}{\sums_{\chi\ppmod{q}}\left|L(\frac12,\chi)\right|^2}.
\]
Using classical results on asymptotic expressions for the first and second moments (see Paley \cite{Pal1931}), we can deduce that the number of $L(\frac12,\chi)\neq 0$ in the family satisfies $\gg\vp^*(q)/\log q$.

By introducing a quantity $M(\chi)$ to mollify the value of $L(\frac12,\chi)$, the mollifier method is employed to recover the logarithmic factor $\log q$, thereby enabling the inference of a positive non-vanishing proportion. Let
\[
\kappa=\frac{1}{\vp^*(q)}\sums_{\substack{\chi\ppmod{q}\\L(\frac{1}{2},\chi) \neq 0}}1.
\]
By applying the Cauchy-Schwarz inequality, we obtain
\begin{align}\label{eqratio}
	\kappa\ge \frac{1}{\vp^*(q)}\frac{{\left|\sums_{\chi\ppmod{q}} L(\frac{1}{2},\chi)M(\chi)\right|^2}}{\sums_{\chi\ppmod{q}} {\left| L(\frac{1}{2},\chi)M(\chi) \right|^2}}.
\end{align}
Balasubramanian and Murty \cite{balasubramanian1992zeros} were the first to obtain a positive proportion using this approach. However, the ordinary mollifier is not efficient, and their proportion is very small.

The first major breakthrough in attaining a considerable proportion was achieved by Iwaniec and Sarnak \cite{Iwaniec1999}. They proved $\kappa\ge \frac13-\ve$ by
using the mollifier
\begin{align*}
	M(\chi)=\sum_{m\leq M} \frac{a(m) \chi(m)}{m^{\frac{1}{2}}},
\end{align*}
where $M=q^{\theta}$ is the length of the mollifier and the coefficients are
\begin{align}
	 a(m)=\mu(m)\frac{\log(y/m)}{\log(y)}.\label{ymdedingyi}
\end{align}
More precisely, they established
\begin{equation} \label{eqtheta1}
	\kappa \ge \frac{\theta}{1+\theta},
\end{equation}
which yields $\kappa\ge\frac13-\ve$ by taking $\theta\rightarrow \frac12$.

For more than a decade, the proportion established by Iwaniec and Sarnak remained the best known result. By introducing a `better' two-piece mollifier, Bui \cite{Bui} surpassed the $\frac13$ record and improved the proportion to $\kappa>0.3411$. In his work \cite{Bui}, Bui commented that ``There are two different approaches to improve the results in this and other problems involving mollifiers. One can either extend the length of the Dirichlet polynomial or use some `better' mollifiers. The former is certainly much more difficult.'' The optimality of mollifiers is discussed in depth by \v{C}ech and Matom\"{a}ki in their recent work \cite{CM2025}. Subsequent developments, including the results presented in this paper, have adopted the former, more difficult approach.

\subsection{The mollifier method of Michel and VanderKam}
A fundamental contribution in this area is the work of Michel and VanderKam \cite{Michel}.
They introduced the following twisted mollifier
\[
M(\chi)=\sum\limits_{m\leq M} \frac{a(m) \chi(m)}{m^{\frac{1}{2}}}+\frac{{\bar{\tau}_\chi}}{q^{\frac{1}{2}}}\sum\limits_{m\leq M} \frac{a(m) {\bar{\chi}}(m)}{m^{\frac{1}{2}}},
\]
where $M=q^{\theta}$ for some $\theta>0$, and $\tau_\chi$ is the Gauss sum. With this new mollifier, they proved
\[
\kappa \ge \frac{2\theta}{2\theta+1}.
\]
Compared to \eqref{eqtheta1}, this result significantly improved the bound for the same $\theta$.
Unfortunately, the length $\theta$ of Michel--VanderKam's mollifier is constrained by the difficulties in treating the following exponential sum:
\begin{align}\label{eqexponential}
\sum_{m_1, m_2\le M}\alpha_{m_1}\beta_{m_2}
\sums_{\substack{n_1,n_2\\m_1m_2n_1>1}}V\left(\frac{n_1n_2}{q}\right)
e\left(\frac{n_2\overline{{n_1}{q_1}{m_1}{m_2}}}{q_2}\right),
\end{align}
where $V$ is a smooth function that restricts $n_1n_2\ll q^{1+\ve}$.
They successfully addressed this exponential sum for $\theta<\frac14$, which resulted in the same proportion $\kappa>\frac13-\ve$ as that achieved by Iwaniec and Sarnak \cite{Iwaniec1999}. Although Michel and VanderKam's work \cite{Michel} primarily focuses on the non-vanishing of derivatives of Dirichlet $L$-functions, an outstanding contribution lies in their introduction of a new feasible approach to extend the range of $\theta$.

\subsection{New inputs by Khan, Mili\'cevi\'c, and Ngo}
For prime moduli $p$, Khan and Ngo \cite{Khan2016} were the first to successfully extend the length of Michel--VanderKam's mollifier. They obtained the proportion $\kappa>\frac38-\ve$ by proving $\theta<\frac3{10}$ is admissible.
A key innovation in their work lies in their treatment of the exponential sum \eqref{eqexponential} through an estimate of a quartic sum of Kloosterman sums, established using the theory of ``trace functions'' in a work of Fouvry et al. \cite{Fouvry2014}. Specifically, they utilized the bound
\begin{align}\label{eqFkloosterman}
\sum_{h\ppmod{p}}S(h,b_1;p)S(h,b_2;p)S(h,b_3;p)S(h,b_4;p)\ll p^{\frac52}
\end{align}
which holds if there is a $b_i$ distinct from the others modulo $p$, where
\begin{equation}
	S(a,b;q)=\sums_{x\ppmod{q}}e_{q}\left(ax+b\bar{x}\right) \nonumber
\end{equation}
denotes the Kloosterman sum. In a subsequent work \cite{KMN2022}, they further improved the proportion to $\frac5{13}$ in collaboration with Mili\'cevi\'c.
A more flexible mollifier of the form
\begin{align}\label{mollifier1}
	M(\chi)=c_1\sum\limits_{m\leq M}\frac{a(m) \chi(m)}{m^\frac{1}{2}}+c_2\frac{\bar{\tau}_\chi}{p^\frac{1}{2}}\sum\limits_{m\leq R}\frac{a(m) \overline{\chi}(m)}{m^\frac{1}{2}}
\end{align}
was employed, where $M\le p^{\theta_1}$ and $R\le p^{\theta_2}$. The values of $\theta_1$ and $\theta_2$ are subject to the analysis of the exponential sum \eqref{eqexponential} but need not be equal, and $c_1$ and $c_2$ are freely chosen coefficients. In their work \cite{KMN2022}, Khan, Mili\'cevi\'c, and Ngo chose $c_1:c_2=\theta_1:\theta_2$ with $\theta_1=\frac38-\ve$ and $\theta_2=\frac14-\ve$.

\subsection{Our main results}
No analogous improvements have been achieved for general moduli. The theory of ``trace functions'' is particularly effective for prime moduli, and as a consequence, the estimate \eqref{eqFkloosterman} for a quartic sum of Kloosterman sums remains unknown in the case of general moduli. This represents the primary challenge.
Inspired by the series of works \cite{Khan2016}, \cite{KMN2022}, and \cite{Michel}, we introduce new techniques to extend the length of Michel--VanderKam's mollifier for general moduli $q$. Our approach involves establishing an average estimate for a quartic sum of Kloosterman sums (see Theorem \ref{le:1}), where we aim to achieve savings by averaging over $b_i$ for general moduli. This relies on delicate estimates for various exponential sums in Section.\ref{secG}--\ref{secTh}. Unlike the fixed $\theta_1>\theta_2$ in \cite{KMN2022}, our final choices of $\theta_1$ and $\theta_2$ are not fixed values but depend on $q$. Furthermore, we ultimately select larger $\theta_2$ such that $\theta_2>\theta_1$ to ensure an adequate range of $b_i$ for achieving the desired savings.
\begin{theorem}\label{1.1}
For sufficiently large integers $q$, there are at least $\frac{7}{19}-\ve$ of primitive Dirichlet characters $\chi$ (mod q) for which $L(s,\chi)\neq0$.
\end{theorem}

We note that in a recent work by Leung \cite{Leung}, a non-vanishing proportion of $0.359$ is achieved for square-free and smooth moduli.
Our proportion in Theorem \ref{1.1} improves upon the previous best known result of $0.3411$ for general moduli, as well as the proportion in Leung's recent work \cite{Leung} for square-free and smooth moduli. However, it does not surpass the proportion for prime moduli. To feel the strength of Theorem \ref{1.1}, we can consider its corresponding $\theta$ in Iwaniec and Sarnak's mollifier. For the proportion in Theorem \ref{1.1}, one needs to take $\theta=\frac7{12}$ in Iwaniec and Sarnak's mollifier, which is even larger than the best-known length $\frac47$ in studying zeros of Riemann zeta-function on the critical line, thanks to the celebrated work of Conrey \cite{Con1989}.

The non-vanishing of Dirichlet $L$-functions has attracted significant attention, and there have been many recent developments in this area. We highlight a few of them. Bui et al. \cite{BPRZ2020} succeeded in breaking the $\frac12$ barrier for the mollifier in the work of Iwaniec and Sarnak \cite{Iwaniec1999}, leading to several notable applications, although the improvement for the non-vanishing problem itself is very slight. Under the unlikely assumption of the existence of an exceptional Dirichlet character, Bui et al. \cite{BPZ2021} proved the non-vanishing proportion $\frac12-\ve$ for prime moduli, which was later improved to $1-\ve$ by \v{C}ech and Matom\"{a}ki \cite{CM2024} under the same assumption.
Under the Generalized Riemann Hypothesis, the classical barrier $\frac12$ was broken by Drappeau, Pratt, and Radziwi\l\l \cite{DPR2023}, and they obtained the proportion $0.51$ by computing the one-level density.

Also based on the mollifier method, David, Faveri, Dunn, and Stucky \cite{DFDS2024} studied the non-vanishing problem for the family of Hecke $L$-functions associated to primitive cubic characters defined over the Eisenstein quadratic number field, and a positive proportion was first obtained.

We present our average result on a quartic sum of Kloosterman sums in the following theorem, which may be of independent interest.
Denote by
\begin{equation}
	G_{q}(b_1,b_2,b_3,b_4)={\sum_{h\ppmod{q}}S(h,\bar{b}_1;q)S(h,\bar{b}_2;q)S(h,\bar{b}_3;q)S(h,\bar{b}_4;q)}, \nonumber
\end{equation}
where $\bar{b}$ is the inverse of $b$ modulo $q$.
\begin{theorem} \label{le:1}
Let $q$ be a large integer. For $B\le q$, we have
	\begin{equation}
		\mathrm{A}(q):={\sums_{1\leq b_1,b_2,b_3,b_4\leq B}}\vert{G}_{q}(b_1,b_2,b_3,b_4)\vert\ll B^4q^{\frac52}\mathring{q}^{-\frac12}+B^2{q}^{3}\mathring{q}+B{q}^{\frac72}{\mathring{q}}^\frac12,
	\end{equation}
where $\mathring{q}$ is the product of all distinct primes coming from odd power of primes in $q$ such that
\begin{align}\label{eqqring}
\mathring{q}=\prod_{\substack{p^{2k+1}\Vert{q}\\k\geq1}}p.
\end{align}
\end{theorem}

\subsection{Notation}
Throughout the paper, we adopt the standard convention that $\ve$ denotes an arbitrarily small positive constant, which may vary from line to line, and the implied constants in all estimates may depend on $\ve$. For a set $A$, we denote by $\#{A}$ the cardinality of $A$, and $\mathbf{1}_{A}$ represents the indicator function of $A$. We use $q^*$ (or $\rho^*, d^*$) to denote the largest square-free factor of an integer $q$ (or $\rho, d$), and $\mathring{q}$ is another square-free factor defined as in \eqref{eqqring}. For brevity, the 4-tuple
$(b_1,b_2,b_3,b_4)$ is sometimes denoted by
$\vec{b}$.

We frequently use the two summation notations
\[
\sums_{n} \quad \text{and}\quad \sums_{x\ppmod{q}},
\]
where, following the standard conventions in analytic number theory, the first sum runs over integers $n$ coprime to $q$, and the second sum runs over all primitive Dirichlet characters $\chi$ modulo $q$.


\section{The proof of Theorem $1$}
We employ a mollifier of the same form as $\eqref{mollifier1}$, defined by
\begin{align}\label{mollifier}
	M(\chi)= c_1\sum_{m\le M} \frac{ a(m) \chi(m)}{m^{\half}}+c_2\frac{\bar{\tau}_{\chi}}{q^\half}\sum_{m\le R } \frac{a(m) \bar{\chi}(m)}{m^{\half}},
\end{align}
where $M=q^{\theta_1}$, $R=q^{\theta_2}$ for some $\theta_1, \theta_2>0$, and $c_1,c_2>0$.
By \eqref{eqratio}, it suffices to evaluate the first and second mollified moments
\begin{align}
	\label{yijieju}\mathcal{M}_1=&\frac{2}{\vp^*(q)} \sump_{\chi\ppmod{q}} L(\thalf,\chi) M(\chi),\\
	\label{Second-moment}\mathcal{M}_2=&\frac{2}{\vp^*(q)} \sump_{\chi\ppmod{q}}\vert L(\thalf,\chi) M(\chi)\vert^2.
\end{align}
We provide a detailed analysis of the moments over even primitive characters, and analogous treatments yield the same result for moments over odd primitive characters. We also assume $q\not\equiv 2 \pmod4$, a convention maintained throughout the rest of the paper, since no primitive characters modulo $q$ exist when $q\equiv 2 \pmod4$.

\begin{proposition}\label{pro:1}
	For $q\not\equiv 2 \pmod4$ and $\theta_1, \theta_2\in(0,\frac12-\ve)$, we have
	\begin{equation}\label{eqfm}
		\mathcal{M}_1=c_1+c_2+O(q^{-{\ve}}).
	\end{equation}
\end{proposition}
\begin{proof}
This asymptotic formula can be derived from the results in \cite{Iwaniec1999} and \cite{Michel}, and a detailed exposition can be found, for example, in \cite[Section 3.1]{KMN2022}. 
\end{proof}
\begin{proposition}\label{pro:2}
	For $q\not\equiv 2 \pmod4$, let $\mathring{q}=q^\gamma$ with $0\leq\gamma\leq\frac13$. We have
	\begin{align}
		\mathcal{M}_2=\frac{{c_1}^2}{\theta_1}+\frac{{c_2}^2}{\theta_2}+{(c_1+c_2)}^2+O\left(\frac{\log\log q}{\log q}\right)
	\end{align}
when $\theta_1,\theta_2\in(0,\frac12-\ve)$ satisfy
\begin{align}\label{eqtheta12}
4\theta_1+6\theta_2<3+\gamma,\ \ \ 2\theta_1+\theta_2<1-\gamma,\ \ \ 4\theta_1<1-\gamma.
\end{align}
\end{proposition}
The proof of Proposition \ref{pro:2} is deferred to Section \ref{secpro:2}. At the end of this section, we present the proof of Theorem \ref{1.1} based on these two propositions.
\begin{proof}[Proof of Theorem $\ref{1.1}$]
After substituting Propositions \ref{pro:1} \& \ref{pro:2} into \eqref{eqratio},
the problem reduces to maximizing the ratio
\begin{align*}
	\Bigg(\left(\frac{c_1}{c_1+c_2}\right)^2{\theta_1}^{-1}+\left(\frac{c_2}{c_1+c_2}\right)^2{\theta_2}^{-1}+1\Bigg)^{-1}
\end{align*}
with suitable coefficients $c_1$ and $c_2$.
By the Cauchy-Schwarz inequality, the ratio attains its maximum
\begin{equation*}
	\left(1+\frac1{\theta_1+\theta_2}\right)^{-1}
\end{equation*}
when ${c_1}:{c_2}={\theta_1}:{\theta_2}$.

Furthermore, for $\theta_1, \theta_2$ in the range \eqref{eqtheta12}, an analysis based on linear optimization theory demonstrates that the maximum can be achieved by choosing $\theta_1$ and $\theta_2$ as follows.
If $0\leq\gamma\leq\frac15$, we set
\begin{align*}
	\theta_1=\frac{1-\gamma}{4}-\ve,\ \ \ \theta_2=\frac{1+\gamma}{3}-\ve
\end{align*}
to obtain
\begin{equation*}
	\left(1+\frac1{\theta_1+\theta_2}\right)^{-1}=\frac{7+\gamma}{19+\gamma}-\ve\ge\frac{7}{19}-\ve.
\end{equation*}
If $\frac15<\gamma\leq\frac13$, we set
\begin{align*}
	\theta_1=\frac{3-7\gamma}{8}-\ve,\ \ \ \theta_2=\frac{1+3\gamma}{4}-\ve
\end{align*}
to obtain
\begin{equation*}
	\left(1+\frac1{\theta_1+\theta_2}\right)^{-1}=\frac{5-\gamma}{13-\gamma}-\ve\ge\frac{7}{19}-\ve.
\end{equation*}
This completes the proof of Theorem \ref{1.1}.
\end{proof}
\section{The mollified second moment}\label{secpro:2}
The treatment of the mollified second moment follows the approach developed in the series of works \cite{Khan2016}, \cite{KMN2022}, and \cite{Michel}, with our new contribution being the application of Theorem \ref{le:1} to extend the length of the mollifier. Upon substituting the mollifier $\eqref{mollifier}$ into $\eqref{Second-moment}$, we arrive at
\[
\mathcal{M}_2=\mathcal{M}_{21}+\mathcal{M}_{22},
\]
where
\begin{align}
	\mathcal{M}_{21}&:=\frac{2c_1^2}{\vp^*(q)} \sump_{\chi\ppmod{q}} \vert L(\thalf,\chi)\vert^2 \Big\vert\sum_{m\le M} \frac{a(m) \chi(m)}{m^{\half}}\Big\vert^2 + \frac{2c_2^2}{\vp^*(q)} \sump_{\chi\ppmod{q}} \vert L(\thalf,\chi)\vert^2 \Big\vert\sum_{m\le R } \frac{a(m) \chi(m)}{m^{\half}}\Big\vert^2,\nonumber\\
	\mathcal{M}_{22}&:=\frac{4c_1c_2}{\vp^*(q)}\sump_{\chi\ppmod{q}}\vert L(\tfrac12,\chi)\vert^2 \frac{{\tau}_{\chi}}{q^\half} \sum_{\substack{m_1 \le M\\ m_2\le R}}\frac{a(m_1)a(m_2) \chi(m_1)\chi(m_2)}{(m_1m_2)^{\half}}.\nonumber
\end{align}

From Iwaniec and Sarnak \cite[Equations (5.5) \& (6.16)]{Iwaniec1999}, we derive the following asymptotic formula
\[
\frac{2}{\vp^*(q)} \sump_{\chi\ppmod{q}} \vert L(\thalf,\chi)\vert^2 \Big\vert\sum_{m\le q^\theta} \frac{a(m) \chi(m)}{m^{\half}}\Big\vert^2=1+\theta^{-1} +O\left(\frac{\log\log q}{\log q}\right)
\]
for $\theta<\frac12$. Consequently, we have
\begin{equation}
	\mathcal{M}_{21}={c_1}^2(1+\theta_1^{-1})+{c_2}^2(1+\theta_2^{-1})+O\left(\frac{\log\log q}{\log q}\right).
\end{equation}

The length of the mollifier is limited by the evaluation of $\mathcal{M}_{22}$. Michel and VanderKam \cite{Khan2016} deduced an asymptotic for $\mathcal{M}_{22}$ with $M,R\ll q^{\frac14-\ve}$. For the case of prime moduli, Khan and Ngo \cite{Khan2016} extended the admissible length to $M,R\ll p^{\frac3{10}}$, and later, in collaboration with Mili{\'c}evi{\'c}, they further extended the range to $M\ll p^{\frac38-\ve}$ and $R\ll p^{\frac14-\ve}$.

To evaluate $\mathcal{M}_{22}$ we need the following approximate functional equation (see, for example, \cite[Equation (3)]{Michel} or \cite[Equation (2-2)]{Khan2016})
\begin{align}\label{L2}
	\vert L(\thalf,\chi)\vert^2= 2\sum_{n_1,n_2\ge 1} \frac{\chi(n_1)\bar{\chi}\left(n_2\right)}{(n_1n_2)^\half}V\left(\frac{n_1n_2}{q}\right),
\end{align}
where
\begin{align*}
	V(x)=\frac{1}{2\pi i} \int_{(2)} \frac{\Gamma(\frac{s}{2}+\frac{1}{4})^2}{\Gamma(\frac{1}{4})^2}(\pi x)^{-s} \frac{ds}{s}
\end{align*}
is a smooth function, which decays rapidly for $x\gg q^\ve$, and is approximately equal to $1$ for $x<1$.
After applying the approximate functional equation, we expand $\mathcal{M}_{22}$ as a series and further eliminate the character using the identity (see \cite[Equation (17)]{Michel}): for $\left(m,q\right)=1$,
\begin{align}
	\sump_{\chi}\tau_{\chi}\chi(m)=\sum_{\substack{q_1q_2=q\\(q_1,q_2)=1}}\mu^{2}(q_1)\varphi(q_2)\cos\(2\pi\frac{\overline{mq_1}}{q_2}\),
\end{align}
to obtain
\begin{align}\label{eqSseries}
\mathcal{M}_{22}=&\frac{8c_1c_2}{\vp^*(q)}\sum_{\substack{q_1q_2=q\\(q_1,q_2)=1}}\mu^{2}(q_1)\varphi(q_2)\\
&\times\sums_{\substack{n_1,n_2\\m_1\le M\\m_2\le R}}\frac{a(m_1)a(m_2)}{(qn_1n_2m_1m_2)^\frac12}V\(\frac{n_1n_2}{q}\)\cos\(2\pi\frac{n_2\overline{n_1m_1m_2q_1}}{q_2}\).\notag
\end{align}
It is shown in \cite[Section 6.1]{Michel} that the main term on the right-hand side of \eqref{eqSseries} arises from the contribution of the terms with $m_1m_2n_1=1$, which yields
\[
2c_1c_2+O\(q^{-\ve}\).
\]

For the remaining terms with $m_1m_2n_1>1$, we employ a dyadic partition of unity to reduce the problem to bounding the following sum: for any complex  $(\alpha_{m_1})_{m_1\in[M_1,2M_1]}$ and $(\beta_{m_2})_{m_2\in[M_2,2M_2]}$ satisfying $\alpha_{m_1}\ll m_1^\ve$, $\beta_{m_2}\ll m_2^\ve$,
\begin{align}
		\mathcal{B}(M_1,M_2,N_1,N_2)=&\frac1{q^\frac32(M_1M_2N_1N_2)^\half}\sum_{\substack{q_1q_2=q\\ \left(q_1,q_2\right)=1}}\mu^{2}(q_1)\varphi(q_2)\sum_{m_1, m_2}\alpha_{m_1}\beta_{m_2}\nonumber\\
&\times \sums_{\substack{n_1,n_2\\m_1m_2n_1>1}}V\left(\frac{n_1n_2}{q}\right)W_1\left(\frac{n_1}{N_1}\right) W_2\left(\frac{n_2}{N_2}\right)
e\left(\frac{n_2\overline{{n_1}{m_1}{m_2}{q_1}}}{q_2}\right)\nonumber
\end{align}
with $N_1N_2\le q^{1+\ve}, M_1 \le \frac{M+1}{2}, M_2 \le \frac{R+1}{2}$, where $W_1$ and $W_2$ are smooth functions compactly supported on the interval $[\frac12,2]$.

The summation over $n_2$ is a normal exponential sum, and we may expect a cancelation when $N_2$ is large.
\begin{lemma}\label{le:2}
	Let $q$ be a positive integer. For $M_1M_2< q^{1-\ve}$, we have
	\begin{align*}
		\mathcal{B}(M_1,M_2,N_1,N_2) \ll q^\ve \Big(\frac{M_1M_2 N_1}{qN_2}\Big)^{1/2} + q^{-\ve}.
	\end{align*}
\end{lemma}
\begin{proof}
	This is given by \cite[Equation (27)]{Michel}.
\end{proof}

If $N_1$ is larger than $q^{\frac12}$, we can exploit the Weil bound to save in the sum over $n_1$. Combining this with Lemma \ref{le:2} allows one to achieve $\theta<\frac14$. To further extend $\theta$,
we deduce a new bound based on Theorem \ref{le:1}.
\begin{lemma}\label{eb1}
	Let $q$ be a positive integer. For $M_1M_2\ll q^{2-\ve}$, we have
	\begin{align*}
\mathcal{B}(M_1,M_2,N_1,N_2) \ll q^\ve\Bigg(\left(\frac{M_1N_2}{N_1}\right)^\frac14+\frac{M_1N_2}{N_1}\Bigg) \left(\({q\mathring{q}}\)^{-\frac18}{M_2}^{\half}+{\mathring{q}}^{\frac14} +\({q\mathring{q}}\)^{\frac18}{M_2}^{-\frac14}\right)+ q^{-\ve}.
	\end{align*}
\end{lemma}
\begin{proof}
We loosely follow the strategy of \cite[Lemma 3.2]{Khan2016} or \cite[Lemma 4.3]{KMN2022}.
Separating $n_1$ into residue classed modulo $q_2$ and then applying the Poisson summation formula, we obtain
\begin{align}\label{Eq2}
\mathcal{B}(M_1,M_2,N_1,N_2)&=\frac{N_1^{\frac12}}{q^\frac32(M_1M_2N_2)^\half}\sum_{\substack{q_1q_2=q\\ \left(q_1,q_2\right)=1}}\frac{\mu^{2}(q_1)\varphi\left(q_2 \right)}{q_2}\\
&\sum_{m_1,m_2}\alpha_{m_1}\beta_{m_2}\sums_{n_2}\sum_k W_2\left(\frac{n_2}{N_2}\right)F(k)S(kn_2, \overline{m_1m_2q_1}; q_2),\notag
\end{align}
where
	\begin{equation*}
		F(k)=\int_{-\infty}^{\infty}W_1(u)V\left(\frac{uN_1n_2}{q}\right)e\left(-\frac{kN_1u}{q_2}\right)du
	\end{equation*}
	decays rapidly when $k\gg{q_2}^{1+\ve}/{N_1}$.
	
For $k=0$, the Kloosterman sum reduces to a Ramanujan sum. Trivially estimating all other sums shows that the contribution of the terms with $k=0$ is not exceeding $O\left(\frac{\left(M_1M_2\right)^\half}{q^{1-\ve}}\right)$, which is acceptable.
	
For the terms with $k\neq0$, we combine variables by setting $h=kn_2\overline{m_1q_1}$. By H\"older inequality, its contribution does not exceed
\begin{equation*}
		\frac{N_1^\half}{q^\frac32(M_1M_2N_2)^\half}\sum_{\substack{q_1q_2=q\\ \left(q_1,q_2\right)=1}}\frac{\mu^{2}(q_1)\varphi\left(q_2 \right)}{q_2}\Bigg(\sum_{h\ppmod {q_2}}\nu(h)^{\frac43}\Bigg)^{\frac34}\Bigg(\sum_{h\ppmod {q_2}}\left|\sum_{m_2} \beta_{m_2} S\left(h,\overline{m}_2;q_2\right)\right|^4\Bigg)^{\frac14},
	\end{equation*}
where $\nu(h)$ denotes the number of representations of writing $h$ as $kn_2\overline{m_1q_1}$ modulo $q_2$. Note that
\begin{align*}
\sum_{h\ppmod {q_2}}\nu(h)^{\frac43}\ll \sum_{h\ppmod {q_2}}\nu(h)^2\ll \sum_{km_1n_2\equiv k'm_1'n_2' \ppmod{q_2}}1\ll \frac{M_1N_2q_2^{1+\ve}}{N_1}\(1+\frac{M_1N_2}{N_1}\).
\end{align*}
Expanding the $4$-th power of the sum of Kloosterman sums yields
\[
\sum_{h\ppmod {q_2}}\left|\sum_{m_2} \beta_{m_2} S\left(h,\overline{m}_2;q_2\right)\right|^4\ll M_2^\ve\sums_{1\leq b_1,b_2,b_3,b_4\leq 2M_2}\left\vert G_{q_2}(b_1,b_2,b_3,b_4)\right\vert.
\]
We handle the right-hand side using Theorem \ref{le:1}. Combining these results and observing that $\mathring{q}_2=\mathring{q}$ (since $q=q_1q_2$ with $q_1$ square-free and $(q_1,q_2)=1$), we conclude the Lemma.
\end{proof}
By virtue of Lemmas \ref{le:2} \& \ref{eb1}, we conclude this section by giving the proof of Proposition \ref{pro:2}, and the rest of this paper is devoted to the proof of Theorem \ref{le:1}.
\begin{proof}[Proof of Proposition \ref{pro:2}]
From the discussions above, it suffices to prove
\begin{align}\label{eqBve}
\mathcal{B}(M_1,M_2,N_1,N_2) \ll q^{-\ve}
\end{align}
for $M_1\ll M\ll q^{\theta_1}$, $M_2\ll R\ll q^{\theta_2}$, and $N_1N_2\ll q^{1+\ve}$.
The bound in Lemma \ref{le:2} indicates \eqref{eqBve} when $\frac{N_2}{N_1}\gg\frac{M_1M_2}{q^{1+\ve}}$. For the other case, our choice of $\theta_1, \theta_2$ in the range \eqref{eqtheta12} ensures
\[
\frac{M_1N_2}{N_1}\ll \frac{M_1^2M_2}{q^{1+\ve}}\ll 1.
\]
Then, by Lemma \ref{eb1}, we have
\begin{align*}
	\mathcal{B}(M_1,M_2,N_1,N_2) \ll q^{-\frac38+\ve}{M}^{\frac12}{R}^{\frac34}{\mathring{q}}^{-\frac18}+q^{-\frac14+\ve}{M}^{\frac12}{R}^{\frac14}{\mathring{q}}^{\frac14}+q^{-\frac18+\ve}{M}^{\frac12}{\mathring{q}}^{\frac18}.
\end{align*}
Given $\theta_1, \theta_2$ in \eqref{eqtheta12}, \eqref{eqBve} follows immediately.
\end{proof}
\section{Initial treatment of $G_{q}$}\label{secG}
We now proceed to the proof of Theorem $\ref{le:1}$. First, we establish preliminary estimates for $G_q$.
Owing to the multiplicativity of Kloosterman sums, $G_{q}$ satisfies the multiplicative relation.
\begin{lemma}
For $q=q_1q_2$ with $\left( q_1,q_2\right)=1$, we have
\[
G_{q}\left( b_1,b_2,b_3,b_4\right)=G_{q_1}\left( b_1,b_2,b_3,b_4\right)G_{q_2}\left( b_1,b_2,b_3,b_4\right).
\]
\end{lemma}
\begin{proof}
The multiplicativity of Kloosterman sums yields
\[
S(h,\ol{b}_i;q)=S(h\ol{q}_2^2,\ol{b}_i;q_1)S(h\ol{q}_1^2,\ol{b}_i;q_2)
\]
for $i=1,2,3,4$.
This in combination with the Chinese Remainder Theory shows 
\begin{align*}
G_{q}\left( b_1,b_2,b_3,b_4\right)=&
\sum_{h_1\ppmod{q_1}}S(h_1\ol{q}_2^2,\bar{b}_1;q_1)S(h_1\ol{q}_2^2,\bar{b}_2;q_1)S(h_1\ol{q}_2^2,\bar{b}_3;q_1)S(h_1\ol{q}_2^2,\bar{b}_4;q_1)\\
&\times \sum_{h_2\ppmod{q_2}}S(h_2\ol{q}_1^2,\bar{b}_1;q_2)S(h_2\ol{q}_1^2,\bar{b}_2;q_2)S(h_2\ol{q}_1^2,\bar{b}_3;q_2)S(h_2\ol{q}_1^2,\bar{b}_4;q_2)
\end{align*}
Then the lemma follows by making the variable changes $h_1\ol{q}_2^2\rightarrow h_1$ and $h_2\ol{q}_1^2\rightarrow h_s$. 
\end{proof}
By the multiplicativity relation and the Chinese Remainder Theorem, it suffices to consider prime power moduli. Specifically, we analyze the following three cases:
\begin{itemize}
  \item $q=p$;
  \item $q=\rho^{2}$, a perfect square;
  \item $q=\rho^{2}\rho^*$, a product of odd powers of primes.
\end{itemize}

For a prime $p$, we denote by
\begin{align}\label{eqDp}
\mathbb{D}(p)=\{(b_1,b_2,b_3,b_4): \forall i,\ b_i\equiv b_j\ppmod p\ \text{for some} \ j\neq i\}
\end{align}
as the subset of $\vec{b}$, in which no component $b_i$ is distinct from the others modulo $p$.
\begin{proposition} \label{prop:1}
	For $(b_1,b_2,b_3,b_4)\notin \mathbb{D}(p)$, we have
	\begin{equation*}
		{G}_{p}(b_1,b_2,b_3,b_4)\ll p^{\frac52}.
	\end{equation*}
	For $(b_1,b_2,b_3,b_4)\in \mathbb{D}(p)$, we have
	\begin{equation*}
		{G}_{p}(b_1,b_2,b_3,b_4)\ll p^3.
	\end{equation*}
	\begin{proof}
		The first bound is obtained from \cite[Proposition 3.2]{Fouvry2014}, while the second bound follows directly from the Weil bound.
	\end{proof}
\end{proposition}
\begin{proposition} \label{prop:2}
	For $q=\rho^{2}$, we have
	\begin{equation}
		{G}_{q}(b_1,b_2,b_3,b_4)\ll q^{\frac52}{\mathcal{N}}_{1}(\rho),
	\end{equation}
	where
	\begin{equation}\label{eqN1}
		{\mathcal{N}}_{1}(\rho)=\#\left\{s,s_i(\bmod\ {\rho}),(ss_i,\rho)=1;\ \sum_{i=1}^{4}s_i\equiv0\ppmod{\rho},s{s_i}^2\equiv\bar{b}_i\ppmod{\rho}\right\}.
	\end{equation}
\end{proposition}
\begin{proof}
Both sides of \eqref{prop:2} exhibit multiplicativity with respect to $\rho$. It suffices to consider the case $\rho=p^k$ with $k\ge1$.
Expressing the Kloosterman sum associated with $b_i$ as
\begin{equation}
S(h,\bar{b}_i;p^{2k})=\sums_{x_i\ppmod{p^{2k}}}e_{p^{2k}}\left(hx_i+\bar{b}_i\bar{x}_i\right). \nonumber
\end{equation}	
We note that $S(h,\bar{b}_i;p^{2k})=0$ if $p\mid h$ (See \cite[Lemma A.4]{KSWX23}).
For $(h,p)=1$, we set
\begin{equation}
h=(1+rp^k)s,\qquad x_i=(1+r_ip^k)s_i\nonumber
\end{equation}
with $r, r_i, s, s_i$ are defined modulo $p^k$ and satisfy $(ss_i,p)=1$,
then
\begin{equation}
\bar{x}_i=(1-r_ip^k)\bar{s}_i,  \nonumber
\end{equation}
where $\bar{x}_i, \bar{s}_i$ are inverses of $x_i$ and $s_i$ modulo $q$. This yields
\begin{equation}
S(h,\bar{b}_i;p^{2k})=\sums_{s_i\ppmod{p^k}} e_{p^{2k}}\left(ss_i+\bar{b}_i\bar{s}_i\right)\sum_{r_i\ppmod{p^k}}e_{p^{k}}\left(rss_i+r_is{s}_i-r_i\bar{b}_i\bar{s}_i\right).\nonumber
\end{equation}
Substituting this into the expression of $G_{p^{2k}}$, we obtain
\begin{align*}
{G}_{p^{2k}}(b_1,b_2,b_3,b_4)&=\sums_{s\ppmod{p^k}}{\sums_{\substack{s_i\ppmod{p^k}\\ (i=1,2,3,4)}}}e_{p^{2k}}\Bigg(\sum_{i=1}^{4}ss_i+\sum_{i=1}^{4}\bar{b}_i\bar{s}_i\Bigg)\\
&\ \ \ \ \ \ \times\sum_{\substack{r,r_i\ppmod{p^k}\\(i=1,2,3,4)}}e_{p^{k}}\Bigg(rs\sum_{i=1}^{4}s_i+\sum_{i=1}^{4}r_i\left(ss_i-\bar{b}_i\bar{s}_i\right)\Bigg).
\end{align*}
The sums over $r$ and $r_i$ vanish unless $s, s_i$ satisfy the conditions in \eqref{eqN1}; in such case, each sum equals $p^k$. A trivial estimate then yields ${G}_{p^{2k}}\ll p^{5k}\mathcal{N}_1(p^k)$, and the lemma follows from the multiplicative property.
\end{proof}
\begin{proposition}\label{prop:3}
	For $q=\rho^{2}\rho^*$, we have
	\begin{equation}
		{G}_{q}(b_1,b_2,b_3,b_4)\ll q^{\frac{5}{2}}\mathring{q}^{\half}{\mathcal{N}}_{2}(\rho),
	\end{equation}
	where
	\begin{eqnarray}
		\begin{aligned}
			{\mathcal{N}}_{2}(\rho)=\#\Bigg\{u,u_i\left(\bmod\ \rho\right),(uu_i,\rho)=1;\sum_{i=1}^{4}(uu_{i}+\bar{b}_i\bar{u}_i)\equiv0\left(\bmod\ \rho\rho^*\right), uu_i\equiv\bar{b}_i\bar{u}_i\left(\bmod\ \rho\right)\Bigg\}.\nonumber
		\end{aligned}
	\end{eqnarray}
\end{proposition}
	\begin{proof}
Observe the set in ${\mathcal{N}}_{2}$, if $u_i$ belongs to the set, so does $u_i(1+c\rho)$. This closure property extends to $u$ as well. Thus, elements in the set are uniquely determined by their residues modulo $\rho$, even in the presence of a congruence condition modulo $\rho\rho^*$.
By the multiplicative property, it suffices to consider the case $\rho=p^k$ for $k\ge1$.
The Kloosterman sum 		
\begin{equation}
	S(h,\bar{b}_i;p^{2k+1})=\sums_{x_i\ppmod{p^{2k+1}}}e_{p^{2k+1}}\left(hx_i+\bar{b}_i\bar{x}_i\right),\nonumber
\end{equation}
associated with ${G}_{q}$, vanishes whenever $p\mid h$. For $(h,p)=1$,
we write
\begin{equation}
h=\left({1+(w+tp)p^k}\right)u, \ \ \ \ \ \ x_i=\left({1+(w_i+t_ip)p^k}\right)u_i\nonumber
\end{equation}
where $t, t_i, u, u_i$ are defined modulo $p^k$ with $(uu_i,p)=1$, and $w, w_i$ are defined modulo $p$. Consequently, we have
\begin{equation}
\bar{x}_i=\left({1-(w_i+t_ip)p^k+{w_i}^2p^{2k}}\right)\bar{u}_i,\nonumber
\end{equation}
where $\bar{x}_i, \bar{u}_i$ denote the inverses of $x_i$ and $u_i$ modulo $q$ respectively.
Applying these quantities to the Kloosterman sums and inserting them into the expression of $G_q$, we simplify to obtain
\begin{align}
{G}_{q}(b_1,b_2,b_3,b_4)=\sums_{u\ppmod{p^k}}{\sums_{\substack{u_i\ppmod{p^k}\\\left(i=1,2,3,4\right)}}}e_{p^{2k+1}}\Bigg(u\sum_{i=1}^{4}u_i+\sum_{i=1}^{4}\bar{b}_i\bar{u}_i\Bigg)W(u,b_i,u_i)T(u,b_i,u_i)\nonumber
\end{align}
where
\begin{align}
W(u,b_i,u_i)=&\sum_{\substack{w,w_i\ppmod p\\\left(i=1,2,3,4\right)}}e_{p^{k+1}}\Bigg(uw\sum_{i=1}^{4}u_i+\sum_{i=1}^{4}w_i\left(uu_i-\bar{b}_i\bar{u}_i\right)\Bigg) \nonumber \\
&\qquad e_p\Bigg(\sum_{i=1}^{4}{w}_i^2\bar{b}_i\bar{u}_i+uw\sum_{i=1}^{4}u_iw_i\Bigg) \nonumber
\end{align}
and
\begin{equation}
T(u,b_i,u_i)=\sum_{\substack{t,t_i\ppmod{p^k}\\\left(i=1,2,3,4\right)}}e_{p^{k}}\Bigg(tu\sum_{i=1}^{4}u_i+\sum_{i=1}^{4}t_i\left(uu_i-\bar{b}_i\bar{u}_i\right)\Bigg).\nonumber
\end{equation}

By the orthogonality of additive characters, we deduce
\begin{equation*}
T(u,b_i,u_i)=p^{5k}\mathbf{1}_{\mathcal{T}},
\end{equation*}
where
\begin{equation}
\mathcal{T}=\left\{u,u_i(\bmod\ {p^k}),(uu_i,p)=1; \sum_{i=1}^{4}u_i\equiv0\left(\bmod \ {p^k}\right), uu_i\equiv\bar{b}_i\bar{u}_i\left(\bmod\ {p^k}\right)\right\}.\nonumber
\end{equation}	

The exponential terms in $W(u,b_i,u_i)$ require further analysis. To address this, we employ the congruence relation of
$\mathcal{T}$ to separate the variables
$w$ and $w_i$ in the exponents. Given
$\sum_{i=1}^{4}u_i\equiv0\left(\bmod\ {p^k}\right)$ and $uu_i\equiv\bar{b}_i\bar{u}_i\left(\bmod\ {p^k}\right)$, we reformulate $W(u,b_i,u_i)$ as
\begin{align}
W(u,b_i,u_i)=\sum_{w\ppmod p}e_{p}\Bigg(w\sum_{i=1}^{4}\frac{uu_i}{p^k}\Bigg)\sum_{\substack{w_i\ppmod p\\(i=1,2,3,4)}}e_p\Bigg(\sum_{i=1}^{4}\left(\frac{uu_i-\bar{b}_i\bar{u}_i}{p^k}w_i+w_i^{2}uu_i+uwu_iw_i\right)\Bigg).\nonumber
\end{align}
Let $\bar{4}$ (or $\bar{2}$) denote the inverse of $4$ (or $2$) modulo $p$. Introducing a redundant factor
\[
1=e_p\Bigg(\sum_{i=1}^4\bar{4}uu_iw^2\Bigg)
\]
we rewrite the sum over $w_i$ as
\begin{align}
\sum_{\substack{w_i\ppmod p\\(i=1,2,3,4)}}&e_p\Bigg(\sum_{i=1}^{4}\left(\frac{uu_i-\bar{b}_i\bar{u}_i}{p^k}w_i+uu_i\left(w_i+\bar{2}w\right)^{2}\right)\Bigg)\nonumber\\
&=\prod_{i=1}^{4}\sum_{w_i(\bmod p)}e_p\left(uu_i{w_i}^2+\frac{uu_i-\bar{b}_i\bar{u}_i}{p^k}(w_i-\bar{2}w)\right),\nonumber
\end{align}
where $w_i$ and $w$ are fully separated, and each $w_i$-sum reduces to a Gauss sum. Applying this in $W(u,b_i,u_i)$ and setting $\bar{2}\equiv(1+p)/2 \(\bmod p\)$, we obtain
\begin{align}
W(u,b_i,u_i)&=\sum_{w\ppmod p}e_{p}\Bigg(w\sum_{i=1}^{4}\frac{uu_i-\bar{2}\left(uu_i-\bar{b}_i\bar{u}_i\right)}{p^k}\Bigg)\prod_{i=1}^{4}\sum_{w_i\ppmod p}e_p\Bigg(uu_iw_i^{2}+\frac{uu_i-\bar{b}_i\bar{u}_i}{p^k}w_i\Bigg)\nonumber\\
&\ll p^2\left\vert\sum_{w\ppmod p}e_{p}\Bigg(\frac{w}{p^k}\Bigg(\frac{1-p}{2}\sum_{i=1}^{4}uu_i+\frac{1+p}{2}\sum_{i=1}^{4}\bar{b}_i\bar{u}_i\Bigg)\Bigg)\right\vert  \nonumber \\
&\ll{p}^3\mathbf{1}_{\mathcal{W}},\nonumber
\end{align}
where
\begin{equation}
\mathcal{W}=\left\{u,u_i\ppmod{p^k},(uu_i,p)=1; \frac{1-p}{2}\sum_{i=1}^{4}uu_i+\frac{1+p}{2}\sum_{i=1}^{4}\bar{b}_i\bar{u}_i\equiv0\ppmod {p^{k+1}}\right\}.\nonumber
\end{equation}
Thus, we conclude that
\begin{align}\label{eq:k_2}
{G}_{p^{2k+1}}(b_1,b_2,b_3,b_4)&\leq p^{5k+3}\#\Bigg\{u,u_i\left(\bmod\ p^{k}\right),(uu_i,p)=1;\ \sum_{i=1}^{4}u_{i}\equiv0\left(\bmod\ p^{k}\right), \nonumber\\
&\ \ \ \ \ \ \ \ \ \ \ \ \ \ uu_{i}\equiv \bar{b}_i\bar{u}_i\left(\bmod\ p^{k}\right), \sum_{i=1}^{4}\left(uu_{i}+\bar{b}_i\bar{u}_i\right)\equiv0\left(\bmod\ p^{k+1}\right)\Bigg\},\nonumber
\end{align}
and then the proposition follows from the multiplicative property.
\end{proof}

\section{Some analogues of Kloosterman sums}\label{secAK}
In this section, we establish upper bounds for several formal analogues of Kloosterman sums, all of which demonstrate square-root cancellation up to a harmless scaling factor.
\begin{lemma} \label{Est of ExpSums}
Let $q\ge1$ and
\begin{equation}
\mathcal{S}(\gamma,\mu,q)=\sums_{s\ppmod {q}}e_{q}(\gamma{\bar{s}^2}+\mu s). \nonumber
\end{equation}
Then
\begin{align}
\mathcal{S}(\gamma,\mu,q)\ll{(\gamma,\mu,q)}^{\frac12}q^{\frac{1}{2}+\ve}.
\end{align}
\end{lemma}
\begin{proof}
This is a special case of  \cite[Proposition 6]{IwaniecKsum}.
\end{proof}
The following lemma establishes an upper bound for an extension of the Gauss sum exhibiting square-root cancellation; this bound will be utilized in the proof of Lemma $\ref{U}$.
\begin{lemma} \label{ExpSums}
For integers $a_1, a_2$, and $q\ge1$, denote by
\begin{equation}
\mathfrak{S}(a_1,a_2,q)=\sum_{s\ppmod{q}}e_{q}(a_1s+a_2s^2P(sq^*)),\nonumber
\end{equation}
where $P(x)$ is any polynomial with integer coefficients, satisfying $P(0)=1$.
Then we have
\[
\mathfrak{S}(a_1,a_2,q)\ll(a_1,a_2,q)^\half{q}^{\half+\ve}.
\]
\end{lemma}
\begin{proof}
We begin by factoring out the greatest common divisor $(a_1,a_2,q)$ in the exponential term, thereby reducing the problem to the case where $(a_1,a_2,q)=1$. This convention will be maintained throughout the proof.

By the Chinese Remainder Theorem, the sum $\mathfrak{S}$ admits a decomposition over coprime moduli. Specifically, if $q=q_1q_2$ with $(q_1,q_2)=1$, then
\[
\mathfrak{S}(a_1,a_2,q)=\mathfrak{S}(a_1\bar{q}_2,a_2\bar{q}_2,q_1)\mathfrak{S}(a_1\bar{q}_1,a_2\bar{q}_1,q_2),
\]
where
$\bar{q}_i$ denotes the inverse of  ${q}_i$ modulo ${q}_j$ ($i\neq j$). It is worth remarking that the polynomial $P(x)$ may vary across different instances of $\mathfrak{S}$ above. However, this does affect our following analysis, as it is independent of the specific form of $P(x)$. This decomposition reduces the problem to bounding $\mathfrak{S}$ for prime powers $p^k$.

Without loss of generality, we assume
\[
P(x)=1+c_1x+c_2x^2+\dots+c_{k-1}x^{k-1}
\]
for any integers $c_1,\cdots, c_{k-1}$. By canceling out $(s,p^k)$ in the exponential, we have
\[
\mathfrak{S}(a_1,a_2,p^k)=\sum_{k'\le k}\mathfrak{S}^*(a_1,a_2,p^{k'}),
\]
where $\mathfrak{S}^*$ is defined similarly to $\mathfrak{S}$, but with the additional condition $(s,q)=1$. This further reduces the problem to proving
\begin{align}\label{eqmS*}
\mathfrak{S}^*(a_1,a_2,p^k)\ll{p}^{\frac k2},
\end{align}
which we address by considering two cases based on the parity of $k$.

For $k=2l$, we set
\begin{equation}
s=x(1+yp^l)\nonumber
\end{equation}
with $x$ and $y$ modulo $p^l$, and $(x,p)=1$. It is straightforward to verify that
\begin{align*}
\mathfrak{S}^*=&\sums_{x\ppmod {p^l}}e_{p^{2l}}\Bigg(a_1x+a_2x^2+a_2\sum_{i=1}^{2l-1}c_{i}p^{i}x^{i+2}\Bigg)\\
&\ \ \ \ \ \times\sum_{y\ppmod {p^l}}e_{p^{l}}\Bigg(\Bigg(a_1x+2a_2x^2+a_2\sum_{i=1}^{l-1}(i+2)c_{i}p^{i}x^{i+2}\Bigg)y\Bigg),
\end{align*}
where the inner sum over $y$\ $(\bmod\ {p^l})$ vanishes unless $x\in\mathbf{V}$, defined by
\begin{equation}
\mathbf{V}=\left\{x\left(\bmod\  p^l\right), (x,p)=1; a_1x+2a_2x^2+a_2\sum_{i=1}^{l-1}(i+2)c_{i}p^{i}x^{i+2}\equiv0 \left(\bmod\ {p^l}\right)\right\}.\nonumber
\end{equation}
For $(a_1,a_2,p)=1$, Hensel's lemma implies that $\mathbf{V}$ is either empty or contains exactly one element. Furthermore, $(a_1a_2,p)=1$ is necessary for $\mathbf{V}$ to be non-empty. Hence, we have
\begin{equation}\label{yi}
\mathfrak{S}^*\ll p^l\#\mathbf{V}\ll p^{\frac{k}{2}}.
\end{equation}

We now turn to the case of odd integers $k=2l+1$. If $l=0$, $\mathfrak{S}^*$ reduces to a Gauss sum, and \eqref{eqmS*} holds.
For $l\geq1$, we express $s$ as
\begin{equation*}
	s=x\left(1+(z+yp)p^l\right),
\end{equation*}
where $z$ is modulo $p$, and $x, y$ are modulo $p^l$ with $(x,p)=1$. We then derive:
\begin{align}
\mathfrak{S}^*
&\ll\sums_{x\ppmod {p^l}}\left\vert\sum_{z\ppmod{p}}e_p\Bigg(\Bigg(a_1x+a_2x^2+a_2\sum_{i=1}^{l}(i+2)c_{i}p^{i}x^{i+2}\Bigg)p^{-l}z+a_2x^2z^2\Bigg)\right\vert\nonumber\\
&\ \ \ \ \ \ \ \ \ \ \ \times\left\vert\sum_{y\ppmod {p^l}}e_{p^l}\Bigg(\Bigg(a_1x+a_2x^2+a_2\sum_{i=1}^{l-1}(i+2)c_{i}p^{i}x^{i+2}\Bigg)y\Bigg)\right\vert,\nonumber
\end{align}
where the sum over $y$ modulo $p^l$ vanishes unless $x\in\mathbf{V}$. The sum over $z$ then becomes a Gauss sum.
Consequently, we have
\[
\mathfrak{S}^*\ll p^{l}\sums_{\substack{x\ppmod{p^l}\\x\in\mathbf{V}}}(a_{2}x,p)^{\half}p^{\frac12}=p^{\frac{2l+1}{2}}.
\]
Combining this bound with $\eqref{yi}$, we establish the desired inequality \eqref{eqmS*}, thereby completing the proof.
\end{proof}

\begin{lemma}\label{U}	
Let $q\ge1$ and $\mu, \gamma$ be integers. For any $\rho\mid q$ with $3\nmid \rho$, let
\begin{equation}
\mathcal{U}\left(\gamma,\mu,\rho,q\right)=\sums_{\substack{u\ppmod{q}\\ {\mu}u-\gamma\bar{u}^2\equiv0\ppmod{\rho}}} e_{q\rho}\left(\gamma\bar{u}^2+2{\mu}u\right).\nonumber
\end{equation}
Then we have the estimate
\begin{equation}\label{U_1}
	\mathcal{U}\left(\gamma,\mu,\rho,q\right)\ll\left(\gamma,\mu,q\right)^{\half}{q}^{\half+\ve}\rho^{-\half}.
\end{equation}
\end{lemma}
\begin{proof}
Note that in the sum, the exponential term is uniquely determined by the residue of $u$ modulo $q$. To elaborate, suppose $u_1\equiv{u_2}\left(\bmod\ {q}\right)$. Writing $u_1=u_2(1+cq)$ and $\bar{u}_1=\bar{u}_2(1-cq)$, we observe
\begin{align}
e_{q\rho}\left(\gamma\bar{u}_1^2+2\mu{u_1}\right)&=e_{q\rho}\left(\gamma\bar{u}_2^2+2\mu{u_2}\right)e_{\rho} \left(c\left(2\mu{u_2}-2\gamma\bar{u}_2^2\right)\right)\nonumber\\
&=e_{q\rho}\left(\gamma\bar{u}_2^2+2\mu{u_2}\right).\nonumber
\end{align}
By canceling $(\gamma,\mu,\rho)$ in both the exponential term and the congruence relation, we obtain the coprimality condition $(\gamma,\mu,\rho)=1$. Under this condition, the congruence equation $\mu u-\gamma\overline{u}^2\equiv0 \pmod{\rho}$ has no solution unless $(\gamma,\rho)=(\mu,\rho)=1$.
Let $q=q_1q_2$, where $q_1$ is the largest factor satisfying $(q_1,\rho)=1$.
By the multiplicative property,
\[
\mathcal{U}\left(\gamma,\mu,\rho,q\right)=\mathcal{U}\left(\gamma \bar{q}_2,\mu\bar{q}_2,1,q_1\right)\mathcal{U}\left(\gamma\bar{q}_1,\mu\bar{q}_1,\rho,q_2\right),
\]
where
\[
\mathcal{U}\left(\gamma\bar{q}_2,\mu\bar{q}_2,1,q_1\right)=\mathcal{S}(\gamma\bar{q}_2,\mu\bar{q}_2,q_1)\ll \left(\gamma,\mu,q_1\right)^{\half}{q_1}^{\half+\ve}
\]
by Lemma \ref{Est of ExpSums}.
We therefore restrict our analysis to establishing the desired bound \eqref{U_1} for parameters $\rho,q$ satisfying $q^*\mid \rho\mid q$ and $(\gamma\mu,\rho)=1$.

For $(\gamma\mu,\rho)=1$, we define
\[
\Delta=\{\delta \ppmod {\rho },\ (\delta,\rho)=1;\ \mu\delta-\gamma{\overline{\delta}^2}\equiv{0}\ppmod{\rho}\},
\]
and it follows that $\#\Delta\ll \rho^\ve$ for $3\nmid \rho$.
We then have
\begin{align}\label{DU}
\mathcal{U}\left(\gamma,\mu,\rho,q\right)=\sum_{\delta\in\Delta}\sums_{\substack{u\ppmod{q}\\ u\equiv \delta \ppmod{\rho}}} e_{q\rho}\left(\gamma\bar{u}^2+2{\mu}u\right).
\end{align}
By expressing $u$ as
$u=\delta\left(1-x\rho \right)$
with $x$ modulo $q/\rho $, it follows that
\begin{align*}
\bar{u}\equiv\bar{\delta}\left(1+x\rho +(x\rho )^2+...+(x\rho )^{k}\right) \ppmod {q\rho}
\end{align*}
for some $k\ge1$ such that $q\rho \mid \rho ^k$.
Substituting this into $\eqref{DU}$ yields
\begin{align}
\mathcal{U}\left(\gamma,\mu,\rho,q\right)=\sum_{\delta\in\Delta}e_{q\rho}\left(\gamma\bar{\delta}^2+2\mu{\delta}\right)\sum_{x \ppmod{q/\rho }} e_{q\rho}\left(\left(2\gamma\bar{\delta}^2-2\mu{\delta}\right)x\rho +3\gamma\bar{\delta}^2(x\rho )^2P(x\rho )\right)\nonumber
\end{align}
where $P(x)$ is a polynomial  satisfying $P(0)=1$. 
 After canceling out $\rho^2$ in the exponential, we obtain
\begin{align}
\mathcal{U}\left(\gamma,\mu,\rho,q\right)\le\sum_{\delta\in\Delta}\Bigg|\sum_{x \ppmod{q/\rho }} e_{q/\rho}\left(\frac{2\gamma\bar{\delta}^2-2\mu{\delta}}{\rho }x+3\gamma\bar{\delta}^2x^2P(x\rho )\right)\Bigg|.\nonumber
\end{align}
Since $q^*\mid \rho \mid q$, we note that $P(x\rho )$ is also a polynomial in $x(q/\rho)^*$.
Thus, by Lemma $\ref{ExpSums}$, we conclude
\begin{align}
\mathcal{U}\left(\gamma,\mu,\rho,q\right)\ll q^{\frac12+\ve}\rho^{-\frac12},\notag
\end{align}
noting that
\[
\(\frac{\gamma\bar{\delta}^2-\mu{\delta}}{\rho }, \gamma, \frac q\rho\)=(\gamma,\mu,q)=1.
\]
This completes the proof.
\end{proof}

\section{Proof of Theorem $\ref{le:1}$}\label{secTh}
In this section, we establish an upper bound for the average of quadruple product of Kloosterman sum
\[
\mathrm{A}(q)={\sums_{1\leq b_1,b_2,b_3,b_4\leq B}}\vert{G}_{q}(b_1,b_2,b_3,b_4)\vert.
\]
\subsection{Initial treatment of $A(q)$}
We decompose $q$ into three coprime parts as follows:
\[
q=\rho c^2(d^2d^*) \ \ \ \text{or}\ \ \   q=\rho (3c^2)(d^2d^*),
 \]
depending on whether $q$ contains a power $3^{2k+1}$ for $k\ge1$,
where $\rho$ denotes the square-free part, $c^2$ (or $3c^2$) represents the product of all even powers (as well as a possible odd power of 3), and $d^2d^*$ corresponds to the product of the remaining odd prime powers.
Associated with the square-free part $\rho$,
we partition the set of $\vec{b}$ as follows:
\[
\left\{\vec{b}=(b_1,b_2,b_3,b_4)\right\}=\bigcup_{\rho_1\rho_2=\rho}\mathbb{D}(\rho_1,\rho_2),
\]
where
\[
\mathbb{D}(\rho_1,\rho_2)=\left\{\vec{b};\ \vec{b}\in \bigcap_{p\mid\rho_1}\mathbb{D}(p),\ \vec{b}\notin\bigcup_{p\mid \rho_2}\mathbb{D}(p)\right\}
\]
with $\mathbb{D}(p)$ defined as in \eqref{eqDp}.
For any $p\mid \rho_1$ and $\vec{b}\in \mathbb{D}(\rho_1,\rho_2)$, at least one of the following holds:
\[
b_1\equiv b_2 \ppmod{p},\quad b_1\equiv b_3 \ppmod{p}, \quad b_1\equiv b_4 \ppmod{p}.
\]
Let $\rho_{11}$ be the product of all $p$ such that $b_1\equiv b_2 \pmod{p}$, let $\rho_{12}$ be the product of all $p\nmid\rho_{11}$ such that $b_1\equiv b_3 \pmod{p}$, and let $\rho_{13}$ be the product of the remaining primes.
This yields
\[
\mathbb{D}(\rho_1,\rho_2)=\bigcup_{\rho_{11}\rho_{12}\rho_{13}=\rho_1}\Big(\mathbb{D}_1(\rho_{11},\rho_2)\cap \mathbb{D}_2(\rho_{12},\rho_2)\cap \mathbb{D}_3(\rho_{13},\rho_2)\Big),
\]
where
\[
\mathbb{D}_i(\rho_{1i},\rho_2)=\left\{\vec{b}\in\mathbb{D}(\rho_1,\rho_2);\ b_1\equiv b_{1+i} \ppmod{\rho_{1i}}\right\}
\]
for $i=1,2,3$.
Without loss of generality, we may assume
$\rho_{11}\ge \rho_1^{\frac13}$.  By the multiplicative property of
${G}_{q}$, we have
\begin{align}\label{A(q)1}
\mathrm{A}\left(q\right)\ll\sum_{\rho_1\rho_2=\rho}\sum_{\substack{\rho_{11}\mid\ \rho_1\\ \rho_{11}\ge\rho_1^{1/3}}} \sum_{\substack{1\leq b_1,b_2,b_3,b_4\leq B\\\vec{b}\in \mathbb{D}_{1}(\rho_{11},\rho_2)}}\left|G_{\rho_1}G_{\rho_2}G_{c^2(or\  3c^2)}G_{d^2d^*}\left(b_1,b_2,b_3,b_4\right)\right|.
\end{align}

If $\rho_{11}\ge B$, the condition $\vec{b}\in\mathbb{D}_1(\rho_{11},\rho_2)$ implies that $(b_1,b_2,b_3,b_4)$ consists of two equal pairs. A straightforward application of the Weil bound then yields this contribution
\begin{equation}
	\ll{B}^2q^3,
\end{equation}
which is an admissible error term in our theorem.
It remains to treat the case where $\rho_{11}<B$, which we assume for the remainder of this section.
Depending on whether $\rho_{11}$ is greater than $\rho_{1}^{\frac12}$ or not,
we decompose the sum on the right-hand side of \eqref{A(q)2} into two parts:
\[
A(q)\ll A_1(q)+A_2(q),
\]
where $A_1(q)$ corresponds to the sum over $\rho_{1}^{\frac13}\le\rho_{11}<\rho_{1}^{\frac12}$, and $A_2(q)$ corresponds to the remaining case.
Applying Propositions \ref{prop:1}-\ref{prop:3} to \eqref{A(q)1} yields
\begin{align}\label{A(q)2}
\mathrm{A}_1(q)\ll c^{5}d^{5}(d^*)^3&\underset{\rho_1\rho_2= \rho}{\sum}\sum_{\substack{\rho_{11}\mid\rho_1\\ \rho_1^{1/3}\le\rho_{11}<\rho_1^{1/2}}} \rho_1^3\rho_2^{\frac52}\sum_{\eta_1\ppmod{\rho_{11}}}\sum_{\eta_2\ppmod{\rho_{11}}}\\
&{\sum_{\substack{b_1,b_2\\b_1\equiv{b_2}\equiv{\eta_1}\ppmod{\rho_{11}}}}}W\left(\frac{b_1}{B}\right)W\left(\frac{b_2}{B}\right){\sum_{\substack{1\leq b_3,b_4\leq B\\b_3\equiv{b_4}\equiv{\eta_2}\ppmod{\rho_{11}}}}}\mathcal{N}_1(c)\mathcal{N}_2(d),\nonumber
\end{align}
noting that $\mathcal{N}_2(3^k)\le \mathcal{N}_1(3^k)$. The revised structure of $A_2(q)$ is similar to that of $A_1(q)$, but omits the congruence $b_1\equiv b_2 \pmod{\rho_{11}}$. Explicitly,
\begin{align*}
\mathrm{A}_2(q)\ll c^{5}d^{5}(d^*)^3&\underset{\rho_1\rho_2= \rho}{\sum}\sum_{\substack{\rho_{11}\mid\rho_1\\ \rho_{11}\ge\rho_1^{1/2}}} \rho_1^3\rho_2^{\frac52}\sum_{\eta\ppmod{\rho_{11}}}\\
&{\sum_{b_1,b_2}}W\left(\frac{b_1}{B}\right)W\left(\frac{b_2}{B}\right){\sum_{\substack{1\leq b_3,b_4\leq B\\b_3\equiv{b_4}\equiv{\eta}\ppmod{\rho_{11}}}}}\mathcal{N}_1(c)\mathcal{N}_2(d).
\end{align*}
When applying the Poisson summation formula, the condition $b_1\equiv b_2 \pmod{\rho_{11}}$ reduces the zero-frequency component by a factor of $\rho_{11}$ but increases the non-zero frequency components by the same factor. For small $\rho_{11}$, the zero-frequency component dominates, whereas for large $\rho_{11}$, the non-zero frequency components dominate. To balance these effects, we impose distinct conditions in $A_1(q)$ and $A_2(q)$.


\subsection{The bound of $A_1(q)$}
By expressing the congruences
\[
\sum_{i=1}^4 s_i\equiv0 \(\bmod\ {c}\)\quad\text{and}\quad \sum_{i=1}^{4}(uu_{i}+\bar{b}_i\bar{u}_i)\equiv0\left(\bmod\ dd^*\right)
\]
as exponential sums via additive characters and applying the simple relation
\[
\mathbf{1}_{uu_i\equiv\bar{u}_i\bar{b}_i\ppmod{d}}=\sum_{t_i\ppmod{d^*}}
\mathbf{1}_{(1+t_i d)uu_i\equiv\bar{u}_i\bar{b}_i\ppmod{dd^*}},
\]
we deduce
\[
\mathcal{N}_1(c)=\frac{1}{c}\sums_{\substack{s,s_i\ppmod{c}\\s{s_i}^2\equiv\bar{b}_i\ppmod{ c}\\(i=1,2,3,4)}} \sum_{t\ppmod{c}}{e}_{c}\Bigg(t\sum_{i=1}^{4}s_i\Bigg)
\]
and
\[
\mathcal{N}_2(d)=\frac{1}{dd^*}\mathop{\sums_{u,u_i\ppmod{d}} \sum_{t_i\ppmod{d^*}}}_{\substack{(1+t_i d)uu_i\equiv\bar{u}_i\bar{b}_i\ppmod{dd^*}\\(i=1,2,3,4)}} \sum_{v\ppmod{dd^*}} e_{ dd^*}\Bigg(v\sum_{i=1}^{4}uu_{i}(2+t_id)\Bigg).
\]
Substituting these two equations into $A_1(q)$,  we observe complete symmetry in the summations over $b_1$ and $b_2$ (and analogously $b_3$ and $b_4$), which leads to
\begin{align}\label{A_1(q)}
	\mathrm{A}_1(q)\ll c^{4}d^4(d^*)^2&\sum_{\rho_1\rho_2=\rho} \rho_1^{3}\rho_2^{\frac52} \sum_{\substack{\rho_{11}\mid\ \rho_1\\\rho_1^{1/3}\le\rho_{11}<\rho_1^{1/2}}}\sum_{\eta_1\ppmod{\rho_{11}}}\sum_{\eta_2\ppmod{\rho_{11}}}\\
	&\sums_{s\ppmod{c}}\sum_{t\ppmod{c}}\sums_{u\ppmod{d}}\sum_{v\ppmod{dd^*}}{|\mathcal{A}_3|}^2{\vert{\mathcal{A}}_4\vert}^2 \nonumber
\end{align}
where
\begin{align}\label{A-three}
{\mathcal{A}}_3={\sums_{s'\ppmod{c}}}e_{ c}\left(ts'\right){\sums_{u'\ppmod{d}}}e_{ dd^*}\left(2uvu'\right)
{\sum_{t'\ppmod{d^*}}}e_{d^*}\left(uvt'u'\right)\sum_{\substack{b\in\mathbb{Z}\\b\equiv\overline{s{s'}^2}\ppmod{c}\\b\equiv\overline{u{u'}^2}(1-t' d)\ppmod{dd^*}\\b\equiv\eta_1\ppmod{\rho_{11}}}}W\left(\frac{b}{B}\right)
\end{align}
and
\begin{align}\label{A-four}
{\mathcal{A}}_4={\sums_{s'\ppmod{c}}}e_{ c}\left(ts'\right){\sums_{u'\ppmod{d}}}e_{ dd^*}\left(2uvu'\right){\sum_{t'\ppmod{d^*}}}e_{d^*}\left(uvt'u'\right)\sum_{\substack{1\leq b\leq B\\b\equiv\overline{s{s'}^2}\ppmod{c}\\b\equiv\overline{u{u'}^2}(1-t' d)\ppmod{dd^*}\\b\equiv \eta_2\ppmod{\rho_{11}}}}1.
\end{align}

We begin by deducing an upper bound for ${\mathcal{A}}_3$, which is then subsequently substituted into the original expression along with an expansion of
${\mathcal{A}}_4$, ultimately leading to the final bound.
\subsubsection{An upper bound for ${\mathcal{A}}_3$}
Applying the Poisson summation formula to the sum over $b$, we obtain its transformed form
\begin{align*}
\frac{B}{\rho_{11}cdd^*}\sum_{a\in \mathbb{Z}}e_{c}\left(a\overline{ \rho_{11}dd^*}\overline{s{s'}^2}\right)e_{ dd^*}\left(a\overline{\rho_{11}c}(\overline{u{u'}^2}(1-t' d))\right)e_{ \rho_{11}}\left(a\eta_1\overline{cdd^*}\right)\widehat{W}\left(\frac{Ba}{\rho_{11}c dd^*}\right).
\end{align*}
After substituting this into $\eqref{A-three}$, we decompose ${\mathcal{A}}_3$ into two parts as follows:
\begin{equation}\label{A_3}
	{\mathcal{A}}_3={\mathcal{A}}_{31}+{\mathcal{A}}_{32},
\end{equation}
where ${\mathcal{A}}_{31}$ represents the contribution of the zero-frequency components with $a=0$, and ${\mathcal{A}}_{32}$ corresponds to the remaining terms.

For ${\mathcal{A}}_{31}$, the sum over $s'$ reduces to a Ramanujan sum, yielding
\begin{equation}\nonumber
	S_{31}:=\sums_{s'\ppmod{c}}e_{ c }(ts')\ll\left(t, c \right)
\end{equation}
and the sum over $u'$ and $t'$ is
\begin{align}
	U_{31}&:={\sums_{u'\ppmod{d}}}e_{ dd^* }\left(2uvu'\right){\sum_{t'\ppmod{d^*}}}e_{d^* }\left(uvt'u'\right)\nonumber\\
	&=d^* \mathbf{1}_{d^* \mid{v}}{\sums_{u'\ppmod{d}}}e_{ d }\left(2u\frac{v}{d^* }u'\right)\nonumber\\
	&\ll d^* \left(\frac{v}{d^* }, d \right)\mathbf{1}_{d^* \mid{v}}.\nonumber
\end{align}
Then we conclude that
\begin{align}\label{A31}
{\mathcal{A}}_{31}&=\frac{B}{\rho_{11} c  dd^* }S_{31}U_{31}\ll\frac{B}{\rho_{11} c  d }\left(t, c \right)\left(\frac{v}{d^*}, d \right)\mathbf{1}_{d^* |v}.
\end{align}

Consistent with the preceding analysis, let $S_{32}$ denote the sum over $s'$ in ${\mathcal{A}}_{32}$ and let $U_{32}$ denote the sum over $u'$ and $t'$. We observe that
\[
S_{32}=\mathcal{S}\left(a\overline{\rho_{11} dd^* }\bar{s},t, c \right).
\]
The summation over $t'$ differs from the previous case:
\begin{equation}
	\sum_{t'\ppmod{d^*}}\left(t'\left(uvu'-a\overline{\rho_{11} c }\overline{u{u'}^2}\right)\right)=d^* \mathbf{1}_{uvu'-a\overline{\rho_{11} c }\overline{u{u'}^2}\equiv0\ppmod{d^*}},\nonumber
\end{equation}
which implies
\[
U_{32}=d^* \mathcal{U}\left(a\overline{\rho_{11} c }\bar{u},uv, d^*,d \right).
\]
Thus, by Lemma $\ref{Est of ExpSums}$ and Lemma $\ref{U}$, we have
\begin{align}\label{A32}
	{\mathcal{A}}_{32}&=\frac{B}{\rho_{11} c  dd^* }\sum_{a\in\mathbb{Z}}\widehat{W}\left(\frac{Ba}{\rho_{11} c  dd^* }\right)e_{ \rho_{11}}\left(a\eta_1\overline{cdd^*}\right)
	S_{32}U_{32}\\
	&\ll\frac{B}{\rho_{11} c  dd^* }\sum_{a\in \mathbb{Z}}\widehat{W}\left(\frac{Ba}{\rho_{11} c  dd^* }\right)
	\left\vert\mathcal{S}\bigg(a\overline{\rho_{11} dd^* }\bar{s},t, c \bigg)\bigg(d^* \mathcal{U}\big(a\overline{\rho_{11} c }\bar{u},uv, d^*,d \big)\bigg)\right\vert\nonumber\\
	&\ll (cdd^*)^{\frac12}.\nonumber
\end{align}
Substituting equations $\eqref{A31}$ and $\eqref{A32}$ into $\eqref{A_3}$ yields
\begin{equation}\label{A3}
	{\mathcal{A}}_3\ll\frac{B}{\rho_{11} c  d }\left(t, c \right)\left(\frac{v}{{d^* }}, d \right)\mathbf{1}_{d^* |v}+(cdd^*)^{\frac12}.
\end{equation}

\subsubsection{Estimating $A_1(q)$}
After substituting the upper bound $\eqref{A3}$ into $\eqref{A_1(q)}$, we express
\begin{equation}\label{A_1^1}
	\mathrm{A}_1(q)\ll{I_1}+{I_2},
\end{equation}
where
\begin{align*}
	I_1= (Bcdd^*)^2&\sum_{\rho_1\rho_2=\rho} \rho_1^{3}\rho_2^{\frac52}\sum_{\substack{\rho_{11}\mid\ \rho_1\\\rho_1^{1/3}\le\rho_{11}<\rho_1^{1/2}}}\rho_{11}^{-2}\sum_{\eta_1\ppmod{\rho_{11}}}\sum_{\eta_2\ppmod{\rho_{11}}}\\
	&\sums_{s\ppmod{c}}\sum_{t\ppmod{c}}(t,c)^2\sums_{u\ppmod{d}}\sum_{v\ppmod{dd^*}}\mathbf{1}_{d^* |v}\left(\frac{v}{{d^* }}, d \right)^2{\vert{\mathcal{A}}_4\vert}^2
\end{align*}
and
\begin{align*}
	I_2= (cd)^5(d^*)^3&\sum_{\rho_1\rho_2=\rho} \rho_1^{3}\rho_2^{\frac52}\sum_{\substack{\rho_{11}\mid\ \rho_1\\\rho_1^{1/3}\le\rho_{11}<\rho_1^{1/2}}}\sum_{\eta_1\ppmod{\rho_{11}}}\sum_{\eta_2\ppmod{\rho_{11}}}\\
	&\sums_{s\ppmod{c}}\sum_{t\ppmod{c}}\sums_{u\ppmod{d}}\sum_{v\ppmod{dd^*}}{\vert{\mathcal{A}}_4\vert}^2.
\end{align*}
Expanding the squared modulus $|\mathcal{A}_4|^2=\mathcal{A}_4\bar{\mathcal{A}}_4$ yields
\begin{align}\label{eqA42}
	{\vert{\mathcal{A}}_4\vert}^2=&{\sums_{s_1,s_2\ppmod{c }}}e_{ c }\left(t\left(s_1-s_2\right)\right){\sums_{u_1,u_2\ppmod{d}}}e_{ dd^* }\left(2uv\left(u_1-u_2\right)\right)\\
	&\times{\sum_{t_1,t_2\ppmod{d^*}}}e_{d^* }\left(uv\left(t_1u_1-t_2u_2\right)\right) \sum_{\substack{1\leq b_i\leq B\ \left(i=1,2\right)\\b_i\equiv\overline{s{s_i}^2}\ppmod{c}\\b_i\equiv\overline{u{u_i}^2}\left(1-t_i d \right)\ppmod{dd^*}\\b_i\equiv \eta_2\ppmod{\rho_{11}}}}1.\nonumber
\end{align}

Substituting this into $I_1$, the summations over $\eta_2, u_i, s_i, t_i$ collapse trivially since these variables are uniquely determined by $u, s, b_i$. This reduces the integral to
\begin{align*}
	I_1\ll (Bcdd^*)^2&\sum_{\rho_1\rho_2=\rho} \rho_1^{3}\rho_2^{\frac52}\sum_{\substack{\rho_{11}\mid\ \rho_1\\\rho_1^{1/3}\le\rho_{11}<\rho_1^{1/2}}}\rho_{11}^{-1}\sum_{\substack{1\leq b_1,b_2\leq B\\b_1\equiv b_2\ppmod{\rho_{11}}}}\\
	&\sums_{s\ppmod{c}}\sums_{u\ppmod{d}}\sum_{t\ppmod{c}}(t,c)^2\sum_{v\ppmod{dd^*}}\mathbf{1}_{d^* |v}\left(\frac{v}{{d^* }}, d \right)^2.
\end{align*}
A straightforward calculation shows that
\begin{align}\label{eqI1b}
I_1\ll B^2(cd)^5(d^*)^2\sum_{\rho_1\rho_2=\rho} \rho_1^{3}\rho_2^{\frac52}\underset{\substack{\rho_{11}\mid\ \rho_1\\\rho_1^{1/3}\le\rho_{11}<\rho_1^{1/2}}}{\sum}\frac{B^2}{\rho_{11}^{2}}\ll{B}^4q^{\frac52}\mathring{q}^{-\frac12}.
\end{align}

We now proceed to estimate $I_2$. Substituting \eqref{eqA42} into the definition of $I_2$, we evaluate the summations over $t$ and $v$. The sum over $t$ is
\begin{equation}
\sum_{t\ppmod{c }}e_{ c }\left(t\left(s_1-s_2\right)\right)= c \mathbf{1}_{s_1\equiv{s_2} \ppmod{c}},\nonumber
\end{equation}
and the sum over $v$ is
\begin{align}
\sum_{v\ppmod{dd^* }}&e_{ dd^* }\left(uv\left(2u_1+t_1u_1 d -2u_2-t_2u_2 d \right)\right)\nonumber\\&= dd^* \mathbf{1}_{uu_1\left(2+t_1 d \right)\equiv uu_2\left(2+t_2 d \right) \ppmod{dd^*}}.\nonumber
\end{align}
The congruence equations $s_1\equiv s_2\pmod{c}$ and $b_i\equiv\overline{s{s_i}^2}\pmod{c}$ imply
\[
b_1\equiv b_2\ppmod c,
\]
and the congruence equation $uu_1\left(1+\left(1+{t_1} d \right)\right)\equiv{uu_2}\left(1+\left(1+{t_2} d \right)\right)\ \ppmod{dd^* }$,
combined with
$b_i\equiv\overline{u{u_i}^2}\left(1-t_i d \right)\ \ppmod{dd^*}$
yields
\begin{equation*}
	b_1\equiv{b}_2\left(\bmod\  d \right).
\end{equation*}
Since $b_1\equiv b_2\left(\bmod\  c  d \right)$ and $\eta_2, u_i, s_i, t_i$ are uniquely determined by $u, s, b_i$,  we sum over the remaining variables trivially to obtain
\begin{align*}
I_2\ll (cd)^6(d^*)^4\sum_{\rho_1\rho_2=\rho}{\rho_1}^3\rho_2^{\frac{5}{2}}\underset{\substack{\rho_{11}\mid\ \rho_1\\\rho_1^{1/3}\le\rho_{11}<\rho_1^{1/2}}}{\sum}\rho_{11}\sum_{s\ppmod{c}}\sum_{u\ppmod{d}}\sum_{\substack{1\leq b_1,b_2\leq B\\b_1\equiv b_2\ppmod{\rho_{11}cd}}}1.
\end{align*}
Note that the value of $\rho_{11}$ is less than $\rho_1^{\frac12}$ at this point. A straightforward calculation shows that
\begin{align}\label{eqI2b}
I_2\ll(cd)^7(d^*)^4\sum_{\rho_1\rho_2=\rho} \rho_1^{3}\rho_2^{\frac52}\underset{\substack{\rho_{11}\mid\ \rho_1\\\rho_1^{1/3}\le\rho_{11}<\rho_1^{1/2}}}{\sum}\(\frac{B^2}{cd}+\rho_{11}B\)
\ll B^2q^{3}\mathring{q}+B{q}^{\frac72}{\mathring{q}}^\half.
\end{align}
Applying $\eqref{eqI1b}$ and $\eqref{eqI2b}$ in $\eqref{A_1^1}$, we conclude that
\begin{equation}\label{A_1(q)3}
A_1(q)\ll B^4q^{\frac52}\mathring{q}^{-\frac12}+B^2{q}^{3}\mathring{q}+B{q}^{\frac72}{\mathring{q}}^\frac12.	
\end{equation}

\subsection{The bound of $A_2(q)$}
The treatment of $A_2(q)$ is nearly identical to that of $A_1(q)$, with only minor modifications required.
Similarly, we have
\begin{align}\label{A_2(q)1}
	\mathrm{A}_2(q)\ll c^{4}d^4(d^*)^2&\sum_{\rho_1\rho_2=\rho} \rho_1^{3}\rho_2^{\frac52} \underset{\substack{\rho_{11}\mid\ \rho_1\\\rho_{11}\ge\rho_1^{1/2}}}{\sum}\sum_{\eta\ppmod{\rho_{11}}}\\
	&\sums_{s\ppmod{c}}\sum_{t\ppmod{c}}\sums_{u\ppmod{d}}\sum_{v\ppmod{dd^*}}{|\mathcal{A}_3'|}^2{\vert{\mathcal{A}}_4\vert}^2,\notag
\end{align}
where $\mathcal{A}_3'$ differs from $\mathcal{A}_3$ only in the absence of a congruence condition on $\rho_{11}$.
An analogous evaluation yields
\begin{align*}
	{\mathcal{A}}_3'\ll\frac{B}{ c  d }\left(t, c \right)\left(\frac{v}{{d^* }}, d \right)\mathbf{1}_{d^* |v}+(cdd^*)^{\frac12}.
\end{align*}
After substituting this into \eqref{A_2(q)1}, we expand $\vert\mathcal{A}_4'\vert^2$ and treat the sum as before to obtain
\begin{align*}
	\mathrm{A}_2\left(q\right)&\ll B^2(cd)^5(d^*)^2\sum_{\rho_1\rho_2=\rho} \rho_1^{3}\rho_2^{\frac52}\underset{\substack{\rho_{11}\mid\ \rho_1\\\rho_{11}\ge\rho_1^{1/2}}}{\sum}\frac{B^2}{\rho_{11}}+(cd)^7(d^*)^4\sum_{\rho_1\rho_2=\rho} \rho_1^{3}\rho_2^{\frac52}\underset{\substack{\rho_{11}\mid\ \rho_1\\\rho_{11}\ge\rho_1^{1/2}}}{\sum}\(\frac{B^2}{\rho_{11}cd}+B\)\\
	&\ll B^4q^{\frac52}\mathring{q}^{-\frac12}+B^2{q}^{3}\mathring{q}+B{q}^{\frac72}{\mathring{q}}^\frac12.
\end{align*}
Combining this with \eqref{A_1(q)3} establishes the bound in Theorem \ref{le:1}.


\begin{thebibliography}{99}

\bibitem{balasubramanian1992zeros}
R.~Balasubramanian and V.~Kumar Murty,
Zeros of Dirichlet $L$-functions,
Ann. Sci. \'Ecole Norm. Sup. (4) \textbf{25} (1992), no.~5, 567--615.

\bibitem{Bui}
H.~M.~Bui,
Non-vanishing of Dirichlet {$L$}-functions at the central point,
Int. J. Number Theory 8 (2012), no.~8, 1855--1881.

\bibitem{BPRZ2020}
H.~M.~Bui, K.~Pratt, N.~Robles, and A.~Zaharescu,
Breaking the 1/2-barrier for the twisted second moment of Dirichlet $L$-functions,
Adv. Math. 370 (2020), 107175, 40 pp.

\bibitem{BPZ2021}
H.~M.~Bui, K.~Pratt, and A.~Zaharescu,
Exceptional characters and nonvanishing of Dirichlet $L$-functions,
Math. Ann. 380 (2021), no.~1--2, 593--642.

\bibitem{Con1989}
J.~B.~Conrey,
More than two fifths of the zeros of the Riemann zeta function are on the critical line,
J. Reine Angew. Math. 339 (1989), 1--26.

\bibitem{CM2024}
M.~\v{C}ech and K.~Matom\"{a}ki,
A note on exceptional characters and non-vanishing of Dirichlet $L$-functions,
Math. Ann. 389 (2024), no. 1, 987--996.

\bibitem{CM2025}
M.~\v{C}ech and K.~Matom\"{a}ki,
On optimality of mollifiers,
	arXiv:2501.12526.

\bibitem{DPR2023}
S.~Drappeau, K.~Pratt, and M.~Radziwi\l\l,
One-level density estimates for Dirichlet $L$-functions with extended support,
 Algebra Number Theory 17 (2023), no.~4, 805--830.

\bibitem{DFDS2024}
C.~David, A.~Faveri, A.~Dunn, and J.~Stucky,
Non-vanishing for cubic Hecke $L$-functions,
	arXiv:2410.03048.

\bibitem{Fouvry2014}
\'E.~Fouvry, S.~Ganguly, E.~Kowalski, and P.~Michel,
Gaussian distribution for the divisor function and {H}ecke eigenvalues in arithmetic progressions,
Comment. Math. Helv. 89 (2014), no.~4, 979--1014.

\bibitem{IwaniecKsum}
J.~B.~Friedlander and H.~Iwaniec,
Incomplete kloosterman sums and a divisor problem,
Ann. of Math. (2) 121 (1985), no.~2, 319--344.

\bibitem{Iwaniec1999}
H.~Iwaniec and P.~Sarnak,
Dirichlet {$L$}-functions at the central point,
Number theory in progress, Vol. 2 ({Z}akopane-{K}o\'scielisko, 1997), de Gruyter, Berlin, 1999, pp.~941--952.

\bibitem{IS2000}
H.~Iwaniec and P.~Sarnak,
The non-vanishing of central values of automorphic $L$-functions and Landau–Siegel zeros,
Israel J. Math. 120 (2000), no.~1, 155--177.

\bibitem{Khan2016}
R.~Khan and H.~T.~Ngo,
Nonvanishing of Dirichlet $L$-functions,
Algebra Number Theory 10 (2016), no.~10, 2081--2091.

\bibitem{KMN2022}
R.~Khan, D.~Mili{\'c}evi{\'c}, and H.~T.~Ngo,
Nonvanishing of Dirichlet $L$-functions II,
Math. Z. 300 (2022), no.~2, 1603--1613.

\bibitem{KMVa}
E.~Kowalski, P.~Michel, and J.~VanderKam,
Mollification of the fourth moment of automorphic $L$-functions and arithmetic applications,
Invent. Math. 142 (2000), no.~1, 95--151.

\bibitem{KMVb}
E.~Kowalski, P.~Michel, and J.~VanderKam,
Non-vanishing of high derivatives of automorphic $L$-functions at the center of the critical strip,
J. Reine Angew. Math. 526 (2000), 1--34.

\bibitem{KSWX23}
B.~Kerr, I.~Shparlinski, X.~Wu, and P.~Xi,
Bounds on bilinear forms with Kloosterman sums,
J. London Math. Soc. (2) 108 (2023), 578--621.

\bibitem{Leung}
S.-K.~Leung,
Non-vanishing of Dirichlet {$L$}-functions with smooth conductors,
Ramanujan J.  66 (2025), no.~4, 69.

\bibitem{Michel}
P.~Michel and J.~VanderKam,
Non-vanishing of high derivatives of Dirichlet $L$-functions at the central point,
J. Number Theory 81 (2000), no.~1, 130--148.

\bibitem{Pal1931}
R.~E.~A.~C.~Paley,
On the k-analogues of some theorems in the theory of the Riemann $\zeta$-function,
Proc. London Math. Soc. (2) 32 (1931), no.~1, 273--311.


\bibitem{Sou2000}
K.~Soundararajan,
Nonvanishing of quadratic Dirichlet $L$-functions at $s=1/2$,
Ann. of Math. (2) 152 (2000), no.~2, 447--488.















%
%
%
%
%
%
%
%
%
%
%
%
%
%
%
%
%
%
%
%
%
%
%
%
%
%
%
%
%
%
%
%
%
%
%




\end{thebibliography}
\end{document}